\documentclass[english]{amsart}
\usepackage[T1]{fontenc}
\usepackage[latin9]{inputenc}
\usepackage{amstext}
\usepackage{amsthm}
\usepackage{amssymb}

\makeatletter
\numberwithin{equation}{section}
\numberwithin{figure}{section}
\theoremstyle{plain}
\newtheorem{thm}{\protect\theoremname}
\newtheorem*{thm*}{\protect\theoremname}
\theoremstyle{definition}
\newtheorem{defn}[thm]{\protect\definitionname}
\theoremstyle{plain}
\newtheorem{lem}[thm]{\protect\lemmaname}
\theoremstyle{plain}
\newtheorem{prop}[thm]{\protect\propositionname}
\theoremstyle{remark}

\theoremstyle{plain}
\newtheorem{cor}[thm]{\protect\corollaryname}
\theoremstyle{definition}

\newtheorem*{ex*}{\protect\examplename}
\usepackage{color,graphics}

\AtBeginDocument{
  
}

\makeatother

\usepackage{babel}
\providecommand{\corollaryname}{Corollary}
\providecommand{\definitionname}{Definition}
\providecommand{\lemmaname}{Lemma}
\providecommand{\propositionname}{Proposition}
\providecommand{\remarkname}{Remark}
\providecommand{\theoremname}{Theorem}
\providecommand{\examplename}{Example}

\global\long\def\dd{\textup{d}}

\begin{document}

\title{On the Saito basis and the Tjurina Number for Plane Branches}

\author{Y. Genzmer and M. E. Hernandes}
\thanks{This work has been partially supported by the R\'eseau de Coop\'eration France-Br\'esil. M. E. Hernandes was also partially supported by CNPq.}

\maketitle

\begin{abstract} We introduce the concept of good Saito basis for a plane curve and we explore it to obtain a formula for the minimal Tjurina number in a topological class. In particular, we give a lower bound for the Tjurina number in terms of the Milnor number that allow us to present a positive answer for a question of Dimca and Greuel.
\end{abstract}

\begin{center} Mathematics Subject Classification: Primary 14H50; Secondary 14B05, 32S05.\end{center}

\begin{center} keywords: Plane curves, Tjurina number, Saito basis. \end{center}

\section{Introduction.}

Let $S:\{f=0\}$ be a germ of an irreducible analytic plane curve. An important analytic invariant of $S$ is the Tjurina number $\tau(S)=\dim_{\mathbb{C}}\frac{\mathbb{C}\{x,y\}}{(f)+J(f)}$ where $J(f)$ denotes the Jacobian ideal of $f$.\\
In general, the computation of $\tau(S)$ is not easy. For instance,
we can obtain it consider a Gr\"obner basis for the ideal
$(f)+J(f)$, or alternatively, it is possible to compute $\tau$ by
the dimension of $\frac{J(f):\left ( f\right )}{J(f)}$ (see Theorem
1 in \cite{zbMATH01549534}) that is related with the
$\mathbb{C}\{x,y\}$-module $\Omega^1(S)$ of all germs of
$1$-holomorphic forms $\omega\in\mathbb{C}\{x,y\}\dd
x+\mathbb{C}\{x,y\}\dd y$ such that $f$ divides $\omega\wedge\dd f$.
More precisely, according to K. Saito \cite{MR586450}, $\Omega^1(S)$
is freely generated by two elements $\{\omega_1,\omega_2\}$. It will
be shown that $\tau(S)$ can be expressed from, among other
invariants, the codimension of the ideal $(g_1,g_2)$ where
$\omega_i\wedge \dd f=g_if\dd x\wedge\dd y$.\\
If $L$ denotes a topological class of a plane curve - for instance,
given by the characteristic exponents - then the Milnor number
$\mu=\dim_{\mathbb{C}}\frac{\mathbb{C}\{x,y\}}{J(f)}$ is constant
for any $S:\{f=0\}\in L$ and $\tau_{\min}\leq \tau(S)\leq \mu$.
Generically, an element $S\in L$ is such that $\tau(S)=\tau_{\min}$,
so $\tau_{\min}$ can be express using the topological data that
characterizes $L$. Delorme in \cite{Delorme1978}, presented a
formula to compute the generic dimension $d(\beta_0,\beta_1)$ of the
moduli space for an irreducible plane curve with characteristic
exponents $(\beta_0,\beta_1)$. As $d(2,\beta_1)=0$ and
$d(\beta_0,\beta_1)=\frac{(\beta_0-3)(\beta_1-3)}{2}+\left [
\frac{\beta_0}{\beta_1}\right ]-1-\mu+\tau_{\min}$ (see
\cite{laudal-martin-pfister-below}) we can compute the minimal
Tjurina number for this topological class. On the other hand,
Peraire in \cite{Peraire} developed an algorithm to compute
$\tau_{\min}$ by means of a flag of $J(f)$.\\
In this paper we present a way to express the difference $\mu-\tau$ for a singular irreducible plane curve $S$ when $\Omega^1\left(S\right)$ admits a basis $\{\omega_1,\omega_2\}$ of special kind, that we call a \emph{good} Saito basis (see Definition \ref{goodbasis}).\\
More specifically, we present a formula (see Theorem \ref{lem:Formula}) to compute the difference between $\mu(S)-\tau(S)$ and $\mu(\widetilde{S})-\tau(\widetilde{S})$ where $\widetilde{S}$ denotes the strict transform of $S$.\\
If $S$ is generic in $L$, then, according to \cite{YoyoBMS}, $S$ admits a good basis and this fact allows us to obtain a formula to compute $\tau_{\min}$ in $L$ by the sole topological data: the sequence of multiplicities in the canonical resolution or the characteristic exponents for instance. In particular, for irreducible plane curves, we are able to present a
lower bound for the minimum Tjurina number in $L$ in terms of the Milnor number that allow us to give an
affirmative answer to a question of Dimca and Greuel \cite{Dimca-Greuel} about the inequality $4\tau >3\mu$ and obtained simultaneously by Alberich-Carrami\~nana {\it et al.} in \cite{dimca-conj} published in ArXiv a few days before the first version of this paper.\\
The paper is organized as follows. In the section 2 we present some general properties of a Saito basis. The concept of a good Saito basis is introduced in the section 3 and its properties as well. The section 4 is devoted to the formula for the minimal Tjurina number, a lower bound for the Tjurina number using the Milnor number and consequently an answer to the Dimca-Greuel question.
\section{The Saito basis.}
Let $S:\{f=0\}$ be a germ of an analytic plane curve and consider
the $\mathbb{C}\{x,y\}$-module $\Omega^1(S)$ of all germs of
$1$-holomorphic forms $\omega\in\mathbb{C}\{x,y\}\dd
x+\mathbb{C}\{x,y\}\dd y$ such that $f$ divides $\omega\wedge\dd f$.
It is equivalent to require that the foliation induced by $\omega$
lets invariant $S$. Saito in \cite{MR586450} shows that
$\Omega^1(S)$ is a free module of rank $2$ and a basis of
$\Omega^1(S)$ is called a Saito basis.\\
It is not trivial to obtain a Saito basis, but there is a simple
criterion to verify if $\left\{ \omega_{1},\omega_{2}\right\}$ is a
basis for $\Omega^1(S)$ (see Theorem, page 270 in \cite{MR586450}).
\begin{thm*}[Saito criterion] The set $\left\{
\omega_{1},\omega_{2}\right\}$ is a Saito basis for $S:\{f=0\}$ if
and only if $\omega_1\wedge\omega_2=uf\dd x\wedge\dd y$, where $u$
is a unit in $\mathbb{C}\{x,y\}$.
\end{thm*}
This criterion can be interpreted as follows : $\left\{
\omega_{1},\omega_{2}\right\}$ is a basis for $\Omega^1(S)$ if the tangency locus between the two forms reduces to $S$.\\
Below, we present some examples of Saito basis for $S:\{f=0\}$. All of them will illustrate, in the sequel, various sensitivities of the Saito basis with respect to small perturbations of the curve $S$. In the whole article, we will keep the same numbering of the examples for the convenience of the reader.
\begin{ex*}[\textbf{1}]
The simplest case is when $f=y^p-x^q$, that is $S_1:\{f=0\}$ is
quasi-homogeneous. In fact, if $\omega_1=qy\dd x- px\dd y$ and
$\omega_2=\dd f$, then $$\omega_1\wedge\omega_2=pqf\dd x\wedge\dd
y$$ and $\{\omega_1,\omega_2\}$ is a basis for $\Omega^1(S_1)$.
\end{ex*}
\begin{ex*}[\textbf{2}]
If $f=y^5-x^6+x^4y^3$ then $S_2:\{f=0\}$ is topologically
quasi-homogeneous, that is, $S_2$ presents characteristic exponents
$(5,6)$, but not analytically equivalent to $y^5-x^6=0$. One can
show that the set $\{\omega_1,\omega_2\}$ where
\begin{eqnarray*}
\omega_1&=& \left ( -6xy+\frac{16}{15}x^3y^2-\frac{8}{5}xy^5\right
)\dd x+\left ( 5x^2+\frac{4}{3}y^3+\frac{4}{5}x^2y^4\right )\dd y\\
\omega_2&=&\left ( -6y^2+\frac{8}{5}x^4-\frac{12}{5}x^2y^3\right )\dd
x+\left ( 5xy+\frac{6}{5}x^3y^2 \right )\dd y
\end{eqnarray*}
satisfy $\omega_1\wedge\omega_2=8f\dd x\wedge\dd y$, so
$\{\omega_1,\omega_2\}$ is a Saito basis for $\Omega^1(S_2)$.
\end{ex*}
\begin{ex*}[\textbf{3}]
The curve $S_3:\{f=0\}$ with $f=y^5-x^{11}+x^6y^3$ is topologically
equivalent to the any curve with characteristic exponents $(5,11)$
and its strict transform is $S_2$. The set $\{\omega_1,\omega_2\}$
where
\begin{eqnarray*}
\omega_1&=& \left ( 605y^2+198xy^3-88x^6\right
)\dd x-\left ( 275xy+66x^2y^2\right )\dd y\\
\omega_2&=&\left ( 605x^4y+150x^5y^2\right )\dd x-\left (
40y^3+275x^5+90x^6y\right )\dd y
\end{eqnarray*}
satisfy $\omega_1\wedge\omega_2=(-24200-7920xy)f\dd x\wedge\dd y$,
so $\{\omega_1,\omega_2\}$ is a Saito basis for $\Omega^1(S_3)$.
\end{ex*}
\begin{ex*}[\textbf{4}]
The class of curve with characteristic exponents the form $(n,n+1)$
has been extensively studied by Zariski \cite{zariski}. The curve
$S_4$ given by $$f=y^7-x^8-7x^6y^2-\frac{147}{8}x^4y^4$$ that,
belongs to the latter class, will be shown of a peculiar interest.
The forms
\begin{eqnarray*}
\omega_1&=& \left(
8x^2y-\frac{147}{8}x^4-\frac{3087}{4}x^2y^2-\frac{21609}{16}y^4\right
)\dd x+\\
&& +\left(
-7x^3+\frac{7}{4}xy^2+\frac{64827}{64}xy^3+\frac{5145}{8}x^3y\right
)\dd y\\
\omega_2&=&\left ( 8xy^2+\frac{1029}{8}x^3y\right )\dd x+\left (-7x^2y+\frac{7}{4}y^3-\frac{1029}{8}x^4 \right )\dd y.
\end{eqnarray*}
produce a Saito basis for $\Omega^1(S_4)$ because
$\omega_1\wedge\omega_2=-\frac{151263}{64}f\dd x\wedge\dd y$.
\end{ex*}
Given a $1$-form $\omega=A\dd x+B\dd y$ we denote by
$\nu(\omega)=\min\{\nu(A),\nu(B)\}$ its algebraic multiplicity,
where $\nu(H)$ indicates the multiplicity of $H\in\mathbb{C}\{x,y\}$
at $(0,0)\in\mathbb{C}^2$.\\
Among all the possible basis $\left\{ \omega_{1},\omega_{2}\right\}$
for $\Omega^1(S)$ we choose some that maximizes the sum
$\nu(\omega_1)+\nu(\omega_2)$ that, following the Saito criterion,
cannot be bigger than $\nu=\nu(f)=\nu\left(S\right)$. For such basis
we denote
\[
\nu_{1}:=\nu\left(\omega_{1}\right)\qquad\nu_{2}:=\nu\left(\omega_{2}\right).
\]
The following result is immediate and identify a new analytical
invariant of $S$.
\begin{prop}
The couple $\left(\nu_{1},\nu_{2}\right)$, up to order, is an
analytical invariant of $S$.
\end{prop}
Remark that the pair $\left(\nu_{1},\nu_{2}\right)$ is not a
topological invariant. For instance, following the examples above, for
$S_1$ with $p=5$ and $q=6$ we have $\left(\nu_{1},\nu_{2}\right)=(1,4)$.
But the curve $S_3$ which is topological equivalent to
$S_1$ has corresponding
pair of multiplicities $(2,2)$.\\
From now on, we consider $S:\{f=0\}$ singular and irreducible (a
plane branch) with a Saito basis $\left\{
\omega_{1},\omega_{2}\right\}$ such that $\omega_{i} =A_{i}\dd
x+B_{i}\dd y$ In particular, we have
\begin{equation}\label{eq:01}
A_{1}B_{2}-A_{2}B_{1} = uf\ \ \ \ \mbox{and}\ \ \ \
A_{i}\frac{\partial f}{\partial y}-B_{i}\frac{\partial f}{\partial
x} = g_{i}f
\end{equation}
where $u(0,0)\neq 0$ and $g_{i}\in\mathbb{C}\{x,y\}$ is called the
cofactor of $\omega_{i}$.\\
Applying a generic linear change of coordinates if necessary, we can
suppose that for $i=1,2$, one has
$\nu\left(A_{i}\right)=\nu\left(B_{i}\right)=\nu_{i} $ and in this
coordinates $\left(x,y\right)$ the tangent cone of $f$, i.e. its
$\nu$-jet, is $f^{(\nu)}=\left(y+\epsilon x\right)^{\nu}$.
\begin{ex*}[\textbf{1}] Consider the irreducible curve $S_1$. Suppose by symmetry that $p<q$, we have
$\nu(A_1)=\nu(B_1)=\nu_1=1$ but $q-1=\nu(A_2)>p-1=\nu(B_2)=\nu_2$.
Consider the change of coordinates $T(x,y)=(x,y-\epsilon x)$ with
$\epsilon\neq 0$ we obtain $f_1=T^*(f)=(y-\epsilon x)^p-x^q$ and the
Saito basis $\eta_1=T^*(\omega_1)$ and $\dd f_1$
\begin{eqnarray*}
\eta_1&=&(q(y-\epsilon x)+\epsilon
px)\dd x-px\dd y\\
\dd f_1&=&(-\epsilon p(y-\epsilon
x)^{p-1}-qx^{q-1})\dd x+p(y-\epsilon x)^{p-1}\dd y
\end{eqnarray*}
 satisfying the above
condition. In addition, $\eta_1\wedge\dd f_1=pqf_1\dd x\wedge\dd y$,
that is, $g_1=pq$ and $g_1=0$.
\end{ex*}
\begin{ex*}[\textbf{2}]
For the curve $S_2$, we have
\[ \omega_1\wedge\dd f=(-30x-8xy^4)f\dd x\wedge\dd y
\ \ \mbox{and}\ \ \omega_2\wedge\dd f=(-30y-12x^2y^2)f\dd x\wedge\dd
y,
\]
that is, $g_1=-30x-8xy^4$ and $g_2=-30y-12x^2y^2$.
\end{ex*}
\begin{ex*}[\textbf{3}]
Considering the curve $S_3$, we have $\nu(A_1)=\nu(B_1)=2$ but
$5=\nu(A_2)>\nu(B_2)=3$. By the change of coordinates
$T(x,y)=(x,x+y)$ we obtain $f_1=T^*(f)=(y+x)^5-x^{11}+x^6(y+x)^3$
and $\eta_i = T^*(\omega_i)  =  \left ( A_i+B_i\right )\dd x+ B_i\dd
y$  with $\nu(A_1+B_1)=\nu(B_1)=2$ and $\nu(A_2+B_2)=\nu(B_2)=3$. In
addition,
\[ \eta_1\wedge\dd f=(3025(x+y)+990x(y+x)^2)f_1\dd x\wedge\dd y
\]
\[ \eta_2\wedge\dd f=(3025x^4+990x^5(y+x))f_1\dd x\wedge\dd
y,
\]
consequently, $g_1=3025(x+y)+990x(y+x)^2$ and
$g_2=3025x^4+990x^5(y+x)$.
\end{ex*}
\begin{ex*}[\textbf{4}]
Finally, for $S_4$ we find
\begin{eqnarray*}
\omega_1\wedge\dd f&=&\left (56x^2-\frac{151263}{16}y^3-\frac{21609}{4}x^2y\right )f\dd x\wedge\dd y\\
\omega_2\wedge\dd f&=&(56xy+1029x^3)f\dd x\wedge\dd y.
\end{eqnarray*}
\end{ex*}

Notice that any generator $\omega_i$ in a Saito basis
$\{\omega_1,\omega_2\}$ has an isolated
singularity, that is, $\gcd(A_i,B_i)=1$. In addition, by
(\ref{eq:01}), we have that $\nu(g_i)\geq \nu_i-1$.
\section{Good Saito basis and the Tjurina number for $S$.}
As we mentioned before, given a Saito basis
$\{\omega_1,\omega_2\}$ for $\Omega^1(S)$ we get $\nu_1+\nu_2\leq
\nu$. In \cite{YoyoBMS}, the first author shows the following theorem:
\begin{thm*}[Generic Basis Theorem] In a fixed
topological class $L$, generically any curve $S$ admits a Saito basis
satisfying
\begin{center}
\begin{tabular}{rl}
$\nu_1=\nu_2=\frac{\nu}{2}$ & if $\nu=\nu(S)$ is even \\ \\
$\nu_1=\nu_2-1=\frac{\nu-1}{2}$ & if $\nu=\nu(S)$ is odd.
\end{tabular}
\end{center}
\end{thm*}
Notice that, generically $\nu_1+\nu_2$ is maximum. Of course,
Example $1$ shows that we can obtain $\nu_1+\nu_2=\nu$
in other cases. This motives the following definition.
\begin{defn}\label{goodbasis}
We say that $S$ (or $\Omega^1(S)$) admits a \emph{good basis} if $\nu_{1}+\nu_{2}=\nu$.
\end{defn}
This section is devoted to present some properties of a good basis.
One of them is related with the index $\mathfrak{i}\left(S\right)$
we introduce in the sequel.\\
Let $E$ be the standard blowing-up of the origin in $\mathbb{C}^{2}$
with coordinates $\left(x,y\right)$ and suppose that, in the chart
$\left(x_{1},y_{1}\right)$ such that
$E\left(x_{1},y_{1}\right)=\left(x_{1},x_{1}y_{1}\right)$, the strict
transform of $S$ goes through $(0,y_{1})$.
\begin{defn} For any $\omega=A\dd x+B\dd y\in\Omega^{1}\left(S\right)$,
we denote by $i\left(\omega\right)\in\mathbb{N}\cup\left\{ \infty\right\}$
the valuation given by
\[
i\left(\omega\right)=\nu_{y_{1}=-\epsilon}\left(A^{\left(\nu\left(\omega\right)\right)}\left(1,y_{1}\right)+y_{1}B^{\left(\nu\left(\omega\right)\right)}\left(1,y_{1}\right)\right)
\]
where $\nu_{y_{1}=-\epsilon}(G)$ denotes de multiplicity of $G\in\mathbb{C}\{y_1\}$ at $-\epsilon\in\mathbb{C}$.\\
Moreover, we denote by $\mathfrak{i}\left(S\right)\in\mathbb{N}$
the integer
\[
\mathfrak{i}\left(S\right)=\min_{\omega\in\Omega^{1}\left(S\right)}i\left(\omega\right).
\]
\end{defn}
The value $i\left(\omega\right)$ is nothing but the index $\rm{Ind}(\mathcal{F},C,0)$ introduced in \cite{camacho} for a germ of foliation $\mathcal{F}$ having $C$ as a smooth invariant curve.\\
Notice that for a given $\omega$, the index $i\left(\omega\right)$
is infinite if and only if $\omega$ is dicritical, that is,
$
A^{\nu\left(\omega\right)}\left(1,y_{1}\right)+y_{1}B^{\nu\left(\omega\right)}\left(1,y_{1}\right)=0.
$\\
However, for any curve $\mathfrak{i}\left(S\right)$ is finite. Indeed,
if $f$ is a reduced equation for $S$ then $\dd f$ belongs to $\Omega^{1}\left(S\right)$
and it is not dicritical, thus $\mathfrak{i}\left(S\right)\leq i\left(\dd f\right)<\infty$. In particular, if $\omega\in\Omega^1(S)$ is non dicritical, then $i(\omega)\leq \nu(\omega)+1$.
\begin{ex*}[\textbf{1}]
For $S_1$ the considered Saito basis is a good basis. Moreover,
$i\left(\omega_1\right)=1$ and $i\left(\omega_2\right)=p$.
\end{ex*}
\begin{ex*}[\textbf{2}]
Having a good basis is a property sensitive to perturbation. Indeed, for instance, the basis of $S_3$ computed in the example is not good, and actually $S_3$ does not admit any good basis. Besides that, we have $i\left(\omega_1\right)=1$ and
$i\left(\omega_2\right)=2$.
\end{ex*}
\begin{ex*}[\textbf{3}]
Good basis is not preserved by blowing-up. In fact, $S_3$ has a good
basis, but its strict transform is analytically equivalent to $S_2$
that does not admit good basis. For $S_3$ we have
$i\left(\omega_1\right)=2$ and $i\left(\omega_2\right)=4$.
\end{ex*}
\begin{ex*}[\textbf{4}]
Finally, $S_4$ does not have a good basis. We find
 $i\left(\omega_1\right)=1$ and $i\left(\omega_2\right)=2$.
\end{ex*}
The next result shows that if $S$ admits a good basis, the index
$\mathfrak{i}\left(S\right)$ is achieved for one of its elements.
\begin{prop}
If $S$ admits a good basis $\left\{ \omega_{1},\omega_{2}\right\}$
then
\[
\mathfrak{i}\left(S\right)=\min\left\{ i\left(\omega_{1}\right),i\left(\omega_{2}\right)\right\} .
\]
\end{prop}
\begin{proof}
By Saito criterion, one has $\omega_{1}\wedge\omega_{2}=uf$ with
$u(0,0)\neq 0$. Since $\nu_{1}+\nu_{2}=\nu,$ one has $
\omega_{1}^{\left(\nu_{1}\right)}\wedge\omega_{2}^{\left(\nu_{2}\right)}\neq0,
$
where $\omega_{i}^{\left(\nu_{i}\right)}=A_{i}^{\nu_{i}}\dd x+B_{i}^{\nu_{i}}\dd y$. In particular, both forms $\omega_1$ and $\omega_2$ cannot
be dicritical and therefore $\min\left\{ i\left(\omega_{1}\right),i\left(\omega_{2}\right)\right\} <\infty.$\\
Now, consider any form $\omega=P_{1}\omega_{1}+P_{2}\omega_{2}\in\Omega^{1}\left(S\right)$ with $P_i\in\mathbb{C}\{x,y\}$ and $m_{i}=\nu\left(P_{i}\right)$. Since $P_{1}^{\left(m_{1}\right)}\omega_{1}^{\left(\nu_{1}\right)}+P_{2}^{\left(m_{2}\right)}\omega_{2}^{\left(\nu_{1}\right)}$
cannot identically vanish, it is the homogeneous part of smallest
degree of $\omega.$ Therefore
\begin{align*}
i\left(\omega\right) & =\nu_{y_{1}=-\epsilon}\left(P_{1}^{\left(m_{1}\right)}\left(1,y_{1}\right)\left(A_{1}^{\nu_{1}}\left(1,y_{1}\right)+y_{1}B_{1}^{\nu_{1}}\left(1,y_{1}\right)\right)\right.\\
 & \qquad\qquad\left.+P_{2}^{\left(m_{2}\right)}\left(1,y_{1}\right)\left(A_{2}^{\nu_{2}}\left(1,y_{1}\right)+y_{1}B_{2}^{\nu_{2}}\left(1,y_{1}\right)\right)\right)\\
 & \geq\min\left\{ i\left(\omega_{1}\right),i\left(\omega_{2}\right)\right\} .
\end{align*}
\end{proof}
In the previous section, we remark that for an element $\omega_i$ in
a Saito basis we get $\nu(g_i)\geq \nu_i-1$ and
$i(\omega_i)\leq\nu(\omega_i)+1$. For good basis it is possible to
obtain the following result.
\begin{lem}\label{i-v1}
Given a good basis $\{\omega_1,\omega_2\}$ for $S$, if
$\nu\left(g_{i}\right)\geq\nu_{i}$ then
$i\left(\omega_{i}\right)=\nu_{i}+1.$
\end{lem}
\begin{proof}
By symmetry let us consider $i=1$ and suppose that
$\nu\left(g_{1}\right)\geq\nu_{1}$. The $(\nu_{1}-1+\nu)$-jet of
\[
A_{1}\frac{\partial f}{\partial y}-B_{1}\frac{\partial f}{\partial
x}=g_{1}f \ \ \ \ \mbox{is}\ \ \ \
A_{1}^{\left(\nu_{1}\right)}\nu\left(y+\epsilon
x\right)^{\nu-1}-B_{1}^{\left(\nu_{1}\right)}\nu\epsilon\left(y+\epsilon
x\right)^{\nu-1}=0,
\]
thus $A_{1}^{\left(\nu_{1}\right)}=\epsilon
B_{1}^{\left(\nu_{1}\right)}$. On the other hand the $\nu$-jet of
$A_{1}B_{2}-A_{2}B_{1}=uf
$
where $u(0,0)\neq 0$ reduces to
$$
A_{1}^{\left(\nu_{1}\right)}B_{2}^{\left(\nu_{2}\right)}-
A_{2}^{\left(\nu_{2}\right)}B_{1}^{\left(\nu_{1}\right)}=
B_{1}^{\left(\nu_{1}\right)}\left(\epsilon
B_{2}^{\left(\nu_{2}\right)}-A_{2}^{\left(\nu_{2}\right)}\right)
=u(0,0)\left(y+\epsilon x\right)^{\nu}.
$$
Thus, there exists some constant $c\neq 0$ such that $
B_{1}^{\left(\nu_{1}\right)}=c\left(y+\epsilon x\right)^{\nu_{1}}$.
Therefore, $\omega_{1}$ can be written
\[
\omega_{1}=\frac{c}{\nu_{1}+1}\dd\left(\left(y+\epsilon
x\right)^{\nu_{1}+1}\right)+ \textup{h.o.t.}
\]
thus $i\left(\omega_{1}\right)=\nu_{1}+1.$
\end{proof}
Notice that the above proof ensures that the inequality
$\nu\left(g_{i}\right)\geq\nu_{i}$ cannot hold for both elements in
a good basis. Moreover, given a good basis for $\Omega^1(S)$ we can
always get a good basis with some nice properties. To do this we
present the following lemmas.
\begin{lem}\label{v-1}
If $\Omega^1(S)$ admits a good basis $\left\{
\omega_{1},\omega_{2}\right\}$, then we can suppose that
$$i\left(\omega_{1}\right)=\mathfrak{i}\left(S\right)\ \ \ \
\mbox{and}\ \ \ \nu\left(g_{1}\right)=\nu_{1}-1.$$
\end{lem}
\begin{proof}
By symmetry we can suppose that
$i\left(\omega_{1}\right)=\mathfrak{i}\left(S\right).$\\
{\bf Case 1.} If
$i\left(\omega_{2}\right)=i\left(\omega_{1}\right)$, then, as
mentioned above, for $i=1$ or $2,$ one has
$\nu\left(g_{i}\right)=\nu_{i}-1$. Switching maybe the two forms, we
can suppose that $\omega_{1}$ satisfies the conclusion of the lemma.\\
{\bf Case 2.} Suppose now that
$i\left(\omega_{1}\right)<i\left(\omega_{2}\right).$\\
{\bf Subcase 2.a} if $\nu_{1}\leq\nu_{2}$, we consider, the family
$\left\{ \omega_{1},\overline{\omega_{2}}\right \}$, where
$\overline{\omega_{2}}=\omega_{2}+c x^{\nu_{2}-\nu_{1}}\omega_{1}$
and $c\in\mathbb{C}$. For a generic value of $c$, we still have a
good basis for $S$. Moreover, the $\nu_{2}$-jet of
$\overline{\omega_{2}}$ is
\[
\left(A_{2}^{\left(\nu_{2}\right)}+c
x^{\nu_{2}-\nu_{1}}A_{1}^{\left(\nu_{1}\right)}\right)\dd
x+\left(B_{2}^{\left(\nu_{2}\right)}+c
x^{\nu_{2}-\nu_{1}}B_{1}^{\left(\nu_{1}\right)}\right)\dd y.
\]
Thus, to evaluate its index, one writes
{\small
\begin{align*}
i\left(\overline{\omega_{2}}\right) & =\nu_{y=-\epsilon}\left(A_{2}^{\left(\nu_{2}\right)}\left(1,y\right)+c A_{1}^{\left(\nu_{1}\right)}\left(1,y\right)+y\left(B_{2}^{\left(\nu_{2}\right)}\left(1,y\right)+c B_{1}^{\left(\nu_{1}\right)}\left(1,y\right)\right)\right) & \\
 & =\nu_{y=-\epsilon}\left(A_{2}^{\left(\nu_{2}\right)}\left(1,y\right)+yB_{2}^{\left(\nu_{2}\right)}\left(1,y\right)+c\left(A_{1}^{\left(\nu_{1}\right)}\left(1,y\right)+yB_{1}^{\left(\nu_{1}\right)}\left(1,y\right)\right)\right) & =i\left(\omega_{1}\right).
\end{align*}
} Thus we are led to the previous case $(1)$.\\
{\bf Subcase 2.b.} Finally, if $\nu_{1}>\nu_{2}$, suppose that
$\nu\left(g_{1}\right)\geq\nu_{1}$, then by Lemma \ref{i-v1} we have
$i\left(\omega_{1}\right)=\nu_{1}+1$. Consequently
$i\left(\omega_{1}\right)>\nu_{2}+1$ and then
$i\left(\omega_{2}\right)>\nu_{2}+1$.  If $\omega_2$ is not
dicritical, the inequality above leads to a contradiction, thus
$\omega_{2}$ is dicritical. Therefore, it can be seen that
$\nu\left(g_2\right)=\nu_2-1$. Let us consider now
$\overline{\omega}_1=\omega_1+ x^{\nu_1-\nu_2}\omega_2$. Then, the
family $\left\{\overline{\omega}_1,\omega_2\right\}$ is still a good
basis and one has
\begin{eqnarray*}
\overline{\omega}_1\wedge\dd f&=&\overline{g}_{1}f\dd x\wedge\dd y\qquad \textup{with }\ \nu(\overline{g}_{1})=\nu_1-1 \\
i\left(\overline{\omega}_1\right)&=&i\left(\omega_1\right)=\mathfrak{i}\left(S\right).
\end{eqnarray*}
\end{proof}
In addition, from a basis for $\Omega^1(S)$ we can get a
basis satisfying the following lemma.
\begin{lem}
\label{lem:We-can-suppose}
Given a basis
$\{\omega_1,\omega_2\}$ for $\Omega^1(S)$ with $i(\omega_1)\leq i(\omega_2)$ we can suppose that
\[ \gcd\left (B_{i},\frac{\partial f}{\partial y}\right )=1,\ \ \ \mbox{for}\ i=1,2.
\]
\end{lem}
\begin{proof}
Suppose that $H=\gcd\left (B_{1},B_{2},\frac{\partial f}{\partial
y}\right )$. Since by (\ref{eq:01})
$
A_{1}B_{2}-A_{2}B_{1}=uf,
$
$H$ would divide $f$. As $\frac{\partial f}{\partial y}$ and $f$ are
relatively prime, we get
\begin{equation}
\gcd\left(B_{1},B_{2},\frac{\partial f}{\partial
y}\right)=1.\label{eq:gcd}
\end{equation}
Now consider the family $\left\{
\overline{\omega_{1}}=\omega_{1}+P_{1}\omega_{2},\overline{\omega_{2}}=\omega_{2}+P_{2}\omega_{1}\right\}
$ where $P_{i}\in\mathbb{C}\left\{ x,y\right\} $ with
$\nu\left(P_{i}\right)\gg1$. Note that for $P_{i}$ of algebraic
multiplicity big enough, the forms
\begin{align*}
\overline{\omega_{1}} & =\left(A_{1}+P_{1}A_{2}\right)\dd x+\left(B_{1}+P_{1}B_{2}\right)\dd y=\overline{A_{1}}\dd x+\overline{B_{1}}\dd y\\
\overline{\omega_{2}} & =\left(A_{2}+P_{2}A_{1}\right)\dd
x+\left(B_{2}+P_{2}B_{1}\right)\dd y=\overline{A_{2}}\dd
x+\overline{B_{2}}\dd y
\end{align*}
satisfy $
\nu\left(\overline{\omega_{i}}\right)=\nu\left(\overline{A_{i}}\right)=\nu\left(\overline{B_{i}}\right)=\nu_{i},
$ and $i(\omega_1)=i(\overline{\omega_{1}})\leq
i(\overline{\omega_{2}})$.\\
Moreover, $\{\overline{\omega_1},\overline{\omega_2}\}$ is a basis
for $\Omega^1(S)$.
Now the relation (\ref{eq:gcd}) ensures that for a generic choice of
the $P_{i}'s, i=1,2$ - in the sense of Krull -, one has
\[
\gcd\left(\overline{B_{i}},\frac{\partial f}{\partial y}\right)=1.
\]
\end{proof}
As a consequence we obtain the following.
\begin{cor}\label{cor1} For any basis $\{\omega_1,\omega_2\}$ for $\Omega^1(S)$ satisfying the previous Lemma
we have
\[ \gcd(B_i,g_i)=\gcd\left ( \frac{\partial f}{\partial
y},g_i\right )=1.
\]
\end{cor}
\begin{proof} As $A_i\frac{\partial f}{\partial y}-B_i\frac{\partial f}{\partial x}=g_if$,
if $1\neq H=\gcd(B_i,g_i)$ then $H$ must divide $A_i\frac{\partial
f}{\partial y}$. By the previous lemma, $\gcd(B_i,\frac{\partial
f}{\partial y})=1$ so $H$ divides $A_i$, a contradiction because
$\omega_i$ has an isolated singularity.\\
Suppose $H'=\gcd\left ( \frac{\partial f}{\partial y},g_i\right )$,
so $H'$ divides $B_i\frac{\partial f}{\partial x}$. As $\gcd\left (
\frac{\partial f}{\partial y},\frac{\partial f}{\partial x}\right
)=\gcd(B_i,g_i)=1$, we must have $H'=1$.
\end{proof}
In particular, the above lemma allow us to consider a good Saito basis $\{\omega_1,\omega_2\}$ with $\mathfrak{i}\left(S\right)=i(\omega_1)$ and $\gcd\left(B_{i},\frac{\partial f}{\partial y}\right)=\gcd(B_i,g_i)=\gcd\left ( \frac{\partial f}{\partial
y},g_i\right )=1$.
\begin{lem}
\label{lem:The-intersection-of} If $S:\{f=0\}$ admits a good basis satisfying the previous conditions, then the intersection of the tangent cone
of
\begin{enumerate}
\item $g_{1}$ and $g_{2}$,
\item $B_{i}$ and $g_{i}$, for $i=1,2$,
\item $B_{i}$ and $\frac{\partial f}{\partial y}$, for $i=1,2$
\end{enumerate}
is empty or equal to $y+\epsilon x=0$.
\end{lem}
\begin{proof}
The $\nu$-jet of (\ref{eq:01}) is
\begin{equation}
A_{1}^{\left(\nu_{1}\right)}B_{2}^{\left(\nu_{2}\right)}-A_{2}^{\left(\nu_{2}\right)}B_{1}^{\left(\nu_{1}\right)}=c\left(y+\epsilon
x\right)^{\nu}.\label{eq:1}
\end{equation}
where $c\neq0$ and $\epsilon\in\mathbb{C}$. Now, for $i=1,2$, both
following relations $ A_{i}^{\left(\nu_{i}\right)}-\epsilon
B_{i}^{\left(\nu_{i}\right)}=0$ cannot be true all together since it
would yield a contradiction with the relation (\ref{eq:1}). Suppose
the relation above is not true for at least $i=1$, then the cofactor
relations ensures that
\[
A_{1}^{\left(\nu_{1}\right)}-\epsilon
B_{1}^{\left(\nu_{1}\right)}=\frac{1}{\nu}g_{1}^{\left(\nu\left(g_{1}\right)\right)}\left(y+\epsilon
x\right).
\]
Combining the above relations yields
$$
g_{1}^{\left(\nu\left(g_{1}\right)\right)}B_{2}^{\left(\nu_{2}\right)}-
g_{2}^{\left(\nu\left(g_{2}\right)\right)}B_{1}^{\left(\nu_{1}\right)}
=c\nu\left(y+\epsilon x\right)^{\nu-1},\ \ \ \mbox{or}\ \ \
g_{1}^{\left(\nu\left(g_{1}\right)\right)}B_{2}^{\left(\nu_{2}\right)}
 =c\nu\left(y+\epsilon x\right)^{\nu-1} $$ from which is derived
$\left(1\right)$ and $\left(2\right)$. The point $\left(3\right)$
follows from the fact that the tangent cone of $\frac{\partial
f}{\partial y}$ and $f$ are the same.
\end{proof}
In what follows we denote by $I_P(G,H)$ the intersection
multiplicity of $G,H\in\mathbb{C}\{x,y\}$ at the point
$P\in\mathbb{C}^2$. If $P=(0,0)$ then we write $I(G,H):=I_P(G,H)$,
that is, $I(G,H)=\dim_{\mathbb{C}}\frac{\mathbb{C}\{x,y\}}{(G,H)}$.\\
An important topological invariant for $S:\{f=0\}$ is the Milnor
number $\mu$ which can be computed by
\begin{equation}\label{milnor}
 \mu:=I\left(\frac{\partial f}{\partial
y},\frac{\partial f}{\partial x}\right)=\sum_{i=1}^{N}\nu_{(i)}(\nu_{(i)}-1)
\end{equation}
where $\nu_{(i)};\ i=1,\ldots ,N$ denote the sequence of multiplicities in
the canonical resolution of $S$. In addition, by Zariski (see (2.4)
in \cite{zariski}), we have
\begin{equation}\label{derivative}
 I\left(\frac{\partial
f}{\partial y},f\right)=\mu+\nu-1.
\end{equation}
Combining the Lemma \ref{v-1} and the above result we can obtain
an expression for $I(g_1,g_2)$.
\begin{lem} If $g_1$ and $g_1$ are the cofactors for a good basis
for $\Omega^1(S)$, then $I(g_1,g_2)$ is finite and
\[
I\left(g_{1},g_{2}\right)=I\left(\frac{\partial f}{\partial
y},B_{1}\right)-I\left(B_{1},g_{1}\right)-\nu+1.
\]
\end{lem}
\begin{proof}
By Lemma \ref{v-1} we have $\nu(g_1)=\nu_1-1<\nu$. As $f$ is
irreducible it follows that $\gcd(f,g_1)=1$ and $I\left
(f\frac{\partial f}{\partial y},g_1\right )<\infty$. So, from
(\ref{eq:01}) that
{\small$$ I\left(f\frac{\partial f}{\partial
y},g_{1}\right) =I\left(A_1B_2\frac{\partial f}{\partial
y}-A_{2}B_1\frac{\partial f}{\partial y},g_1\right)
 =I\left ( B_1B_2\frac{\partial f}{\partial x}-A_{2}B_1\frac{\partial f}{\partial y},g_1\right )
 =I\left(B_1g_2f,g_1\right).
$$} Hence,
\begin{equation}\label{aux1}
I(g_1,g_2)=I\left(\frac{\partial f}{\partial
y},g_{1}\right)-I(B_1,g_1).
\end{equation}
The Corollary \ref{cor1} insures that $\frac{\partial f}{\partial
y}$ and $g_{1}$ are coprime. So, by (\ref{aux1}) and using
(\ref{derivative}) we obtain
\begin{align*}
I\left(g_{1},g_{2}\right) & =I\left(\frac{\partial f}{\partial y},g_{1}\right)
+I\left(\frac{\partial f}{\partial y},f\right) -I\left(\frac{\partial f}{\partial y},f\right) -I\left(B_{1},g_{1}\right)\\
 & =I\left(\frac{\partial f}{\partial y},g_{1}f\right)-I\left(\frac{\partial f}{\partial y},f\right)-I\left(B_{1},g_{1}\right)\\
 & =I\left(\frac{\partial f}{\partial y},A_{1}\frac{\partial f}{\partial y}-B_{1}\frac{\partial f}{\partial x}\right)-\left(\mu+\nu-1\right)-I\left(B_{1},g_{1}\right)\\
 & =I\left(\frac{\partial f}{\partial
 y},B_{1}\right)-\nu+1-I\left(B_{1},g_{1}\right).
\end{align*}
\end{proof}
Let us consider the Tjurina number $\tau$ of a plane curve $S:\{f=0\}$, that is,
\[ \tau:=\dim_{\mathbb{C}}\frac{\mathbb{C}\{x,y\}}{\left (f,\frac{\partial f}{\partial
 y},\frac{\partial f}{\partial
 x} \right )}.
\]
Zariski (see Theorem 1 in \cite{zbMATH03232554}) considered the
torsion submodule $T\Omega^1_{\mathcal{O}/\mathbb{C}}$ of the
K\"ahler differential module $\Omega^1_{\mathcal{O}/\mathbb{C}}$
over $\mathcal{O}=\frac{\mathbb{C}\{x,y\}}{(f)}$ and he showed that
$\tau=\dim_{\mathbb{C}}T\Omega^1_{\mathcal{O}/\mathbb{C}}.
$\\
On the other hand, Michler (Theorem 1 in \cite{zbMATH01549534}) proved that
$T\Omega^1_{\mathcal{O}/\mathbb{C}}$ is isomorphic as
$\mathcal{O}$-module, to $\frac{\left ( \frac{\partial f}{\partial
 y},\frac{\partial f}{\partial x}\right ):\left ( f\right )}{\left (
 \frac{\partial f}{\partial  y},\frac{\partial f}{\partial  x}\right )}$.
 As $\left ( \frac{\partial f}{\partial  y},\frac{\partial f}{\partial  x}\right ):(f)$ is precisely the cofactor ideal of $S$, that is, $(g_1,g_2)$, one has
\[\tau=\dim_{\mathbb{C}}\frac{(g_1,g_2)}{\left (
\frac{\partial f}{\partial y},\frac{\partial f}{\partial x}\right  )}
 =\dim_{\mathbb{C}}\frac{\mathbb{C}\{x,y\}}{\left ( \frac{\partial f}{\partial
 y},\frac{\partial f}{\partial x}\right
 )}-\dim_{\mathbb{C}}\frac{\mathbb{C}\{x,y\}}{(g_1,g_2)}=\mu-I(g_1,g_2),
\]
that is,
\[\mu-\tau=I(g_1,g_2).\]
Denoting $\widetilde{\mu}$ the Milnor number
of $\widetilde{S}$ we provide in the next theorem a precise relation between $\mu-\tau$ and
$\widetilde{\mu}-\widetilde{\tau}$ by means of the analytic invariants
we have introduced previously for curves that admit a good basis.
\begin{thm}
\label{lem:Formula} If $S$ admits a good basis, then
\[
\mu-\tau=\widetilde{\mu}-\widetilde{\tau}
+\left(\nu_{1}-1\right)\left(\nu_{2}-1\right)+\mathfrak{i}\left(S\right)-1.
\]
\end{thm}
\begin{proof}
By symmetry, one can suppose $
\mathfrak{i}\left(S\right)=\min\left\{
i\left(\omega_{1}\right),i\left(\omega_{2}\right)\right\}
=i\left(\omega_{1}\right).$ By Lemma \ref{lem:The-intersection-of}
and the Max-Noether formula one has,
\[
\mu-\tau=I(g_1,g_2)=I_{(0,-\epsilon)}\left(\tilde{g}_{1},\tilde{g}_{2}\right)+\nu\left(g_{1}\right)\nu\left(g_{2}\right),
\]
where $\widetilde{H}:=E^*(H)$ and $E$ denotes the standard
blowing-up of the origin in $\mathbb{C}^2$.\\
In addition, the previous lemma and Lemma
\ref{lem:The-intersection-of}, yield {\small
\begin{align*}
I(g_1,g_2)
 & =I\left(\frac{\partial f}{\partial y},B_{1}\right)-I\left(B_{1},g_{1}\right)-\nu+1\\
 & =I_{(0,-\epsilon)}\left(\tilde{\frac{\partial f}{\partial y}},\tilde{B}_{1}\right)-
 I_{(0,-\epsilon)}\left(\tilde{B}_{1},\tilde{g}_{1}\right)+\nu\left(\frac{\partial f}{\partial y}\right)\nu\left(B_{1}\right)-\nu\left(B_{1}\right)\nu\left(g_{1}\right)-\nu+1.
\end{align*}
} If $\tilde{\omega}_{i}=\frac{E^{*}\omega_{i}}{x^{\nu_{i}}}$, then
the Saito criterion yields $ x^{\nu_{1}}\tilde{\omega}_{1}\wedge
x^{\nu_{2}}\tilde{\omega}_{2}=\widetilde{u}x^{\nu}\tilde{f}x\dd
x\wedge\dd y. $ Since we have a good basis, that is,
$\nu_{1}+\nu_{2}=\nu,$ one has $
\tilde{\omega}_{1}\wedge\tilde{\omega}_{2}=u\tilde{f}x\dd x\wedge\dd
y. $ Locally around $(0,-\epsilon)$ for $i=1,2$ we have
$$
\tilde{\omega}_{i} =\left(A_{i}^{\nu_{i}}\left(1,y\right)+yB_{i}^{\nu_{i}}\left(1,y\right)+x\left(\cdots\right)\right)\dd x+x\left(B_{i}^{\nu_{i}}\left(1,y\right)+\left(\cdots\right)\right)\dd y\\
$$
We notice that the form
\[
\overline{\omega}_{2}=\frac{1}{x}\left(\tilde{\omega}_{2}-\frac{A_{2}^{\nu_{2}}\left(1,y\right)+yB_{2}^{\nu_{2}}\left(1,y\right)}{A_{1}^{\nu_{1}}\left(1,y\right)+yB_{1}^{\nu_{1}}\left(1,y\right)}\tilde{\omega}_{1}\right)
\]
is holomorphic at $\left(0,-\epsilon\right)$ and $\left\{
\tilde{\omega}_{1},\overline{\omega}_{2}\right\} $ is a Saito basis
for $\widetilde{S}:\{\tilde{f}=0\}$. A computation shows that the
cofactor associated to $\tilde{\omega}_{1}$ is written $
g_{1}^{'}=\tilde{g}_{1}+\nu\tilde{B}_{1}. $ Moreover, one has $
\tilde{\omega}_{1}=\left(\tilde{A}_{1}+y\tilde{B}_{1}\right)\dd
x+x\tilde{B}_{1}\dd y=A'\dd x+B'\dd y.
$
Now,
\[
\left(A_{1}^{\nu_{1}}\left(1,y\right)+yB_{1}^{\nu_{1}}\left(1,y\right)+x\left(\cdots\right)\right)\frac{\partial\tilde{f}}{\partial y}-x\tilde{B}_{1}\frac{\partial\tilde{f}}{\partial x}=g_{1}^{'}\tilde{f}.
\]
If $x$ divides $g_{1}^{'}$ then $\tilde{\omega}_{1}$ would be
dicritical and this is not possible. Therefore,
\begin{align*}
I_{(0,-\epsilon)}\left(x,g_{1}^{'}\tilde{f}\right) &
=I_{(0-\epsilon)}\left(A_{1}^{\nu_{1}}\left(1,y\right)+yB_{1}^{\nu_{1}}\left(1,y\right),x\right)+I_{(0,-\epsilon)}\left(x,\frac{\partial\tilde{f}}{\partial
y}\right).
\end{align*}
and, by Corollary \ref{cor1}, $
I_{(0,-\epsilon)}\left(x,g_{1}^{'}\right)=i\left(\omega_{1}\right)-1=\mathfrak{i}\left(S\right)-1.
$\\
Notice that $\tilde{B}_{1}$ and $g_{1}^{'}$ cannot have a common
divisor, since it would be a common divisor of $\tilde{g}_{1}$ and
$\tilde{B}_{1}$ that is impossible by Lemma
\ref{lem:We-can-suppose}. So,
$$
I_{(0,-\epsilon)}\left(\tilde{B}_{1},\tilde{g}_{1}\right)
=I_{(0,-\epsilon)}\left(x\tilde{B}_{1},g_{1}^{'}\right)-\mathfrak{i}\left(S\right)+1
 =I_{(0,-\epsilon)}\left(B_{1}^{'},g_{1}^{'}\right)-\mathfrak{i}\left(S\right)+1.
$$
Moreover,
\begin{align*}
I_{(0,-\epsilon)}\left(\frac{\partial \tilde{f}}{\partial y},\tilde{B}_{1}\right) & =I_{(0,-\epsilon)}\left(\frac{\partial\tilde{f}}{\partial y},B_{1}^{'}\right)-I_{(0,-\epsilon)}\left(\frac{\partial\tilde{f}}{\partial y},x\right)\\
 & =I_{(0,-\epsilon)}\left(\frac{\partial\tilde{f}}{\partial
 y},B_{1}^{'}\right)-I_{(0,-\epsilon)}\left(\tilde{f},x\right)+1.
\end{align*}
So, as $\frac{\partial \tilde{f}}{\partial y}=\tilde{\frac{\partial f}{\partial y}}$ and combining all the above relation yields
\begin{align*}
\mu-\tau & =I_{(0,-\epsilon)}\left(\frac{\partial\tilde{f}}{\partial y},B_{1}^{'}\right)-I_{(0,-\epsilon)}\left(\tilde{f},x\right)+1-\left(I_{y=-\epsilon}\left(B_{1}^{'},g_{1}^{'}\right)-\mathfrak{i}\left(S\right)+1\right)\\
 & \qquad+\nu\left(\frac{\partial f}{\partial y}\right)\nu\left(B_{1}\right)-\nu\left(B_{1}\right)\nu\left(g_{1}\right)-\nu+1\\
 & =I_{(0,-\epsilon)}\left(g_{1}^{'},g_{2}^{'}\right)+\left(\nu-1\right)\nu_{1}-\nu_{1}\nu\left(g_{1}\right)-\nu+\mathfrak{i}\left(S\right).
\end{align*}
As
$I_{(0,-\epsilon)}\left(g_{1}^{'},g_{2}^{'}\right)=\widetilde{\mu}-\widetilde{\tau}$
and $\nu\left(g_{1}\right)=\nu_1-1$, we obtain finally
\[
\mu-\tau=\widetilde{\mu}-\widetilde{\tau}+\left(\nu_{1}-1\right)\left(\nu_{2}-1\right)+\mathfrak{i}\left(S\right)-1.
\]
\end{proof}
Let us analyze the examples previously considered.
\begin{ex*}[\textbf{1}]
For $S_1$ we have a good basis with $\nu_1=1$, $\nu_2=p-1$ and
$\mathfrak{i}\left(S_1\right)=1$, then $\mu-\tau=0$ as classically
known.
\end{ex*}
\begin{ex*}[\textbf{2}]
Notice that for $S_2$ we have
$\mathfrak{i}\left(S_2\right)=i\left(\omega_1\right)=1$,
$\nu_1=\nu_2=2$ and $\widetilde{S}_3$ is regular, so
$\widetilde{\mu}-\widetilde{\tau}=0$. In this way,
\[1=I(g_1,g_2)=\mu-\tau=0+(2-1)(2-1)+1-1.
\]
So, the formula in the previous theorem holds although $S_2$ does
not admit any good basis.
\end{ex*}
\begin{ex*}[\textbf{3}]
For $S_3$ we get
$\mathfrak{i}\left(S_4\right)=i\left(\omega_1\right)=2$, $\nu_1=2,
\nu_2=3$ and $\widetilde{S}_3$ is analytically equivalent to $S_2$,
so $\widetilde{\mu}-\widetilde{\tau}=1$. In this way,
\[4=I(g_1,g_2)=\mu-\tau=1+(2-1)(3-1)+2-1.
\]
\end{ex*}
\begin{ex*}[\textbf{4}]
As we presented above, $S_4$ does not have a good basis. We have
 $\mathfrak{i}\left(S_4\right)=i\left(\omega_1\right)=1$, $\nu_1=\nu_2=3$ and $\widetilde{\mu}-\widetilde{\tau}=0$, but in this case,
\[5=I(g_1,g_2)=\mu-\tau\neq 4=0+(3-1)(3-1)+1-1.\]
A more detailed analysis shows that Lemma \ref{lem:The-intersection-of}
 is not valid in this case because the intersection of the tangent cone
  of $g_1$ and $g_1$ is $x=0$ that is distinct to the tangent cone $y=0$ of $S_4$.
\end{ex*}

\section{The minimal Tjurina number and the Dimca-Greuel question for plane branches. }

Given a curve $S$, we denote by $L=L(S)$ its topological class. Although the Milnor number is constant in $L$, the same is not true for the Tjurina number $\tau(S)$. On the other hand, as $\tau(S)$ is upper semicontinuous, the minimum value $\tau_{\min}$ for curves in $L$ is achieved generically and it should be computed by the sole topological data (see Chapitre III, Appendice of \cite{zariski} by Teissier).\\
For a topological class $L$ given by characteristic exponents $(\beta_0,\beta_1)$, Delorme in \cite{Delorme1978} presented a formula for the dimension of the generic component of the Moduli space that allow us to compute $\tau_{\min}$. For an arbitrary topological class, Peraire (see \cite{Peraire}) presented an algorithm to compute the $\tau_{\min}$ using flag of the Jacobian ideal.\\
In this section, using the last theorem and results of \cite{YoyoBMS}, we give an alternative method to compute $\tau_{\min}$ in a fixed topological class $L$ and as a bonus we are able to answer a question of Dimca-Greuel for the irreducible plane curves.\\
If $S$ admits a good basis we can not insure that the same is valid
for $\widetilde{S}$ (see Example (3)). However, this property is
true generically.
\begin{thm}\label{theo-tau-min}
Let $L$ the topological class of plane branch given by the characteristic exponents $(\beta_0,\beta_1,\ldots ,\beta_s)$, $\tau_{\min}$
the minimal Tjurina number in $L$ and $\widetilde{\tau}_{\min}$
the minimal Tjurina number in $\widetilde{L}$. If $S$ is generic in $L$, then
\begin{equation}\label{formula}
\mu-\tau_{\min}=\widetilde{\mu}-\widetilde{\tau}_{\min}
 +\left ( \left [\frac{\beta_0}{2}\right ]-1\right )\left (\beta_0-\left [\frac{\beta_0}{2}\right ]-1\right )+\mathfrak{i}\left(S\right)-1.
\end{equation}
Moreover, if $n=\left\lceil \frac{\beta_{1}}{\beta_{1}-\beta_0}\right\rceil $, then  $\mathfrak{i}\left(S\right)=\left [ \frac{\beta_0}{2}\right ]+1-p_1(S)$, where $p_1(S)$ can be computed in the following
way:
\[
\bullet\  \mbox{if}\ \beta_0\ \mbox{is even then}\ p_{1}\left(S\right)=\left\{ \begin{aligned}1 & \ \textup{ if }n=2\\
1 & \ \textup{ if \ensuremath{\beta_{1}} is even}\\
\frac{n-1}{2} & \ \textup{ if \ensuremath{\beta_{1}} is odd and \ensuremath{n} odd}\\
\frac{n-2}{2} & \ \textup{ if \ensuremath{\beta_{1}} is odd and \ensuremath{n} is even}
\end{aligned}
\right.
\]
\[
\bullet\  \mbox{if}\ \beta_0\ \mbox{is odd then}\ p_{1}\left(S\right)=\left\{ \begin{aligned}0 & \ \textup{ if }n=2\\
1 & \ \textup{ if \ensuremath{\beta_{1}} is odd}\\
\frac{n-3}{2} & \ \textup{ if \ensuremath{\beta_{1}} is even and \ensuremath{n} odd}\\
\frac{n-2}{2} & \ \textup{ if \ensuremath{\beta_{1}} is even and \ensuremath{n} is even.}
\end{aligned}
\right.
\]
\end{thm}
\begin{proof} Suppose that $\beta_0=\nu(S)$ is even. According to the Generic Basis Theorem, $S$ admits a good basis $\{\omega'_1,\omega'_2\}$ with $\nu(\omega'_{1})=\nu(\omega'_{2})=\frac{\beta_0}{2}$. For generic $\alpha_1,\alpha_2\in \mathbb{C}$
$
\left\{\omega_1=\omega'_1+\alpha_2\omega'_2,\omega_2=\omega'_2+\alpha_1\omega'_1\right\}
$ remain a good basis with $\nu_1=\nu_2=\frac{\beta_0}{2}$ and
$i(\omega_1)=i(\omega_2)$.\\
Now, according to \cite{YoyoBMS} - using the notations of the
mentioned paper, it refers to the case $\delta_{1}=0$ and
$\delta_{2}=1$ - we obtain
$\nu_1+1=\frac{\beta_0}{2}+1=\sum_{q\in\mathbb{P}^1}{\rm
Ind}(\widetilde{\mathcal{F}},C,q)=i(\omega_1)+p_1(S),$ that is,
\[\mathfrak{i}\left(S\right)=i\left(\omega_1 \right)=\nu_1+1-p_1\left(S\right)=\frac{\beta_0}{2}+1-p_1\left(S\right)=\left [\frac{\beta_0}{2}\right ]+1-p_1(S)\]
where $p_{1}\left(S\right)$ is described in \cite{YoyoBMS}.\\
Now, suppose $\beta_0$ is odd and let $\left\{
\omega'_{1},\omega'_{2}\right\}$ be a Saito basis for $S\cup l$ with
$l$ a generic line that without loss of generality can be considered
$x=0$. As $\nu(S\cup l)$ is even, by the same above argument, we can
suppose that
\[\nu(\omega'_1)=\nu(\omega'_2)=\frac{\beta_0+1}{2}=\left
[\frac{\beta_0}{2}\right ]+1\ \ \mbox{and}\ \ \mathfrak{i}(S\cup
l)=i\left(\omega'_1\right)=i\left(\omega'_2\right).\] Denoting
$\omega'_{i}
=\left(a_{i}\left(y\right)+x\left(\cdots\right)\right)\textup{d}x+x\left(\cdots\right)\textup{d}y
$ and considering generic $\alpha_1,\alpha_2\in \mathbb{C}$ we
obtain a good Saito basis
$\{\omega_1=\omega'_{1}+\alpha_2\omega'_{2},\omega_2=\omega'_{2}+\alpha_1\omega'_{1}\}$
such that
$\nu\left(a_{1}(y)+\alpha_2a_2(y)\right)=\nu\left(a_{2}(y)+\alpha_1a_1(y)\right)$,
\[i\left(\omega_{1}\right)=i\left(\omega'_{1}\right)=i\left(\omega'_{2}\right)=i\left(\omega_{2}\right)
\ \mbox{and}\
\nu(\omega_1)=\nu(\omega'_1)=\nu(\omega'_2)=\nu(\omega_2).\] Now the
family
\[
\left\{
\omega_{1},\frac{1}{x}\left(\omega_{2}-\frac{a_{2}(y)+\alpha_1a_1(y)}{a_{1}(y)+\alpha_2a_2(y)}\omega_{1}\right)\right\}
\]
is a good Saito basis for $S$. Finally, since
$i\left(\frac{1}{x}\left(\omega_{2}-\frac{a_{2}(y)+\alpha_1a_1(y)}{a_{1}(y)+\alpha_2a_2(y)}\omega_{1}\right)\right)\geq
i\left(\omega_{1}\right)$, one has $
\mathfrak{i}\left(S\right)=i\left(\omega_{1}\right). $
By the description of $p_1(S\cup l)$ given in \cite{YoyoBMS} - using
the notations of the article, it refers to the case $\delta_{1}=1$
and $\delta_{2}=1$ - we get
\[\mathfrak{i}\left(S\right)=i\left(\omega_1 \right)=\frac{\nu\left(S\right)+1}{2}-p_1\left(S\right)=
\left [\frac{\beta_0}{2}\right ]+1-p_1\left(S\right).\]
Thus, the proof of the formula is a consequence of Theorem \ref{lem:Formula} noticing that by the Generic Basis Theorem we have $\nu_1=\left [\frac{\beta_0}{2}\right ]$ and $\nu_2=\beta_0-\left [\frac{\beta_0}{2}\right ]$.
\end{proof}
\begin{ex*}[\textbf{5}]
In \cite{Peraire}, Peraire computed the minimum Tjurina number for
the topological class whose characteritic exponents are $(9,12,17)$.
After five blowing-ups, we obtain a curve with multiplicity $2$. The
corresponding characteristics exponents of the sequence of blown-up
curves are $(3,14),\ (3,11),\ (3,8),\ (3,5),\ (2,3)$. Applying
inductively the formula (\ref{formula}), one accumulates
contribution to the difference $\mu -\tau_{\textup{min}}$. Actually,
it can be seen that the respective contributions are $15,\ 1,\ 1,\
1,\ 0,\ 0$. Thus $\tau_{\textup{min}}=\mu-18=98-18=80$ which
coincides with the computation of Peraire.
\end{ex*}
The last theorem allow us obtain a formula to compute the minimum Tjurina number in a topological class using the multiplicity sequence.
\begin{cor} Let $L$ a topological class of a singular plane branch determined by the multiplicity sequence $\nu_{(1)},\nu_{(2)},\ldots ,\nu_{(N)},\nu_{(N+1)}=1,\ldots$. The minimal Tjurina number achieved in $L$ is \[\tau_{\min}=\sum_{i=1}^{N} \left (\nu_{(i)}^2+\left [ \frac{\nu_{(i)}}{2}\right ]\left ( \left [\frac{\nu_{(i)}}{2} \right ]-\nu_{(i)}-1\right )-1+p_1(S_{(i)}) \right )
\]
where $S_{(i)}$ denote the curve with multiplicity $\nu_{(i)}$ in the canonical resolution process for a generic curve in $L$.
\end{cor}
\begin{proof}
Applying inductively the formula presented in the last theorem and using that $\mathfrak{i}\left(S_{(i)}\right )=\left [ \frac{\nu_{(i)}}{2}\right ]+1-p_1(S_{(i)})$ yields
\[
\tau_{\min} = \mu-\sum_{i=1}^{N} \left (\left (\left [ \frac{\nu_{(i)}}{2}\right ]-1\right )\left ( \nu_{(i)}-\left [\frac{\nu_{(i)}}{2} \right ]-1\right )-(\mathfrak{i}\left(S_{(i)}\right )-1) \right )
\]
\begin{equation}\label{genericformula}
\hspace{0.5cm}  = \mu +\sum_{i=1}^{N} \left (\left [ \frac{\nu_{(i)}}{2}\right ]\left (\left [\frac{\nu_{(i)}}{2} \right ]-\nu_{(i)}-1\right )+\nu_{(i)}-1+p_{1}(S_{(i)}) \right ).
\end{equation}
As $\mu=\sum_{i=1}^{N}\nu_{(i)}\left ( \nu_{(i)}-1\right )$, we get the proof.
\end{proof}
In \cite{Dimca-Greuel}, Dimca
and Greuel present an interesting question about the Tjurina number for curves in a given topological class $L$. More specifically, they ask if $4\tau(S) >3\mu(S)$ for any curve in $L$.\\
As the Tjurina number is semicontinuous in $L$ and we have obtained a formula for the $\tau_{\min}$, we are able to given a lower bound for the Tjurina number in terms of the Milnor number and it answered positively the previous question for the irreducible case.
\begin{cor}\label{dimca}
Let $S$ be a singular irreducible plane curve. Then
\[
\tau\left(S\right)\geq \frac{3}{4}\mu\left(S\right)+\frac{\sqrt{1+4\mu(S)}-1}{8}.
\] In particular, $4\tau(S) >3\mu(S)$.
\end{cor}
\begin{proof}
We denote $\mu=\mu(S)$.
It is sufficient to show the inequality for the $\tau_{\min}$.\\
By (\ref{genericformula}), the relation below holds
\[ 4\tau_{\min}-3\mu= \mu +4\sum_{i=1}^{N} \left (\left [ \frac{\nu_{(i)}}{2}\right ]\left (\left [\frac{\nu_{(i)}}{2} \right ]-\nu_{(i)}-1\right )+\nu_{(i)}-1+p_{1}(\nu_{(i)}) \right ).
\]
Now, using that $\mu=\sum_{i=1}^{N}\nu_{(i)}\left (
\nu_{(i)}-1\right )$ and $4\left [ \frac{\nu_{(i)}}{2}\right ]\left
(\left [\frac{\nu_{(i)}}{2} \right ]-\nu_{(i)}-1\right
)=-\nu_{(i)}^2-2\nu_{(i)}+\delta_i
$
with $\delta_i =0$ if $\nu_{(i)}$ is even and $\delta_i=3$ if
$\nu_{(i)}$ is odd, we obtain
\[ 4\tau_{\min}-3\mu= \sum_{i=1}^{N} \left ( \nu_{(i)}+\delta_i +4(p_{1}(S_{(i)})-1)\right ).
\]
Now, by Theorem \ref{theo-tau-min} we have that:
\begin{itemize}
\item if $\nu_{(i)}$ is even, then $p_1(S)\geq 1$ and $\nu_{(i)}+0 +4(p_{1}(\nu_{(i)})-1)\geq \nu_{(i)}$,
\item if $\nu_{(i)}$ is odd, then $p_1(S)\geq 0$ and $\nu_{(i)}+3 +4(p_{1}(\nu_{(i)})-1)\geq \nu_{(i)}-1.$
\end{itemize}
So, the following inequality follows
\begin{equation}\label{inequality}
4\tau_{\min}-3\mu\geq \sum_{i=1}^{N}(\nu_{(i)}-1).
\end{equation}
As $\mu=\sum_{i=1}^{N}(\nu_{(i)}-1)^2+\sum_{i=1}^{N}\left (
\nu_{(i)}-1\right )$ we get $4\tau_{\min}-3\mu\geq
\mu-\sum_{i=1}^{N}(\nu_{(i)}-1)^2.$ Using (\ref{inequality}), that
is, $-(4\tau_{\min}-3\mu)^2\leq -\left (
\sum_{i=1}^{N}(\nu_{(i)}-1)\right )^2\leq
-\sum_{i=1}^{N}(\nu_{(i)}-1)^2$, we obtain $4\tau_{\min}-3\mu\geq
\mu-(4\tau_{\min}-3\mu)^2$ and consequently
\[
\tau\left(S\right)\geq\tau_{\min}\geq \frac{3}{4}\mu+\left (\frac{-1+\sqrt{1+4\mu}}{8} \right ).
\]
\end{proof}

\begin{ex*}[\textbf{6}] Let us consider the topological
class $L$ determined by the characteristic exponents $(141,142)$.
The Milnor number for any curve in $L$ is
$\mu=(141-1)(142-1)=19740$. Using the lower bound presented in the
las result we obtain $\tau_{\min}\geq 14840$. For this topological
class it follows by the Delorme result ({\it cf.}
\cite{Delorme1978}) that $\tau_{\min}=14910$.
\end{ex*}

While we submit the first version of this paper to Arxiv, we discover that, at the same time, a positive answer for the Dimca-Greuel question was obtained by Alberich-Carrami\~nana {\it et al.} and published in Arxiv \cite{dimca-conj} a few days before. Although the methods are a bit different, the key ingredient is still the formula for the generic dimension of the moduli space obtained in \cite{YoyoBMS}.

\bibliographystyle{plain}
\bibliography{Bibliographie}

\vspace{1cm}

\begin{tabular}{lcl}
Genzmer, Y. & & Hernandes, M. E. \\
{\it yohann.genzmer$@$math.univ-toulouse.fr} & & {\it mehernandes$@$uem.br}
\end{tabular}

\end{document}